\documentclass[11pt]{amsart}

\usepackage{geometry}
\usepackage{todonotes}

\usepackage{amsfonts, amssymb, amscd}
\numberwithin{equation}{section}

\usepackage[symbol]{footmisc}

\usepackage{bm}
\usepackage{verbatim}
\usepackage{mathrsfs}
\usepackage{graphicx}
\usepackage{tikz-cd}
\usepackage{subcaption}
\usepackage{listings}
\usepackage{subfiles}
\usepackage[toc,page]{appendix}
\usepackage{mathtools}
\usepackage{comment}
\usepackage{enumerate}
\usepackage{enumitem}

\usepackage{graphicx}
\graphicspath{{images/}}

\usepackage{appendix}
\usepackage{hyperref}
\newcommand{\Rr}{\mathbb{R}}

\newtheorem{thm}{Theorem}[section]

\newtheorem{lem}[thm]{Lemma}

\theoremstyle{definition}
\newtheorem{defn}[thm]{Definition}
\newtheorem{ques}[thm]{Question}
\theoremstyle{definition}
\newtheorem{rem}[thm]{Remark}

\newtheorem{ex}[thm]{Example}

\theoremstyle{definition}

\begin{document}

\title{Rational nef and anti-nef polytopes are not uniform}
\author{Lingyao Xie}

\address{Department of Mathematics, The University of Utah, Salt Lake City, UT 84112, USA}
\email{lingyao@math.utah.edu}

\date{\today}

\begin{abstract}
We give two examples which show that rational nef and anti-nef polytopes are not uniform even for klt surface pairs, answering a question of Chen-Han. We also show that rational nef polytopes are uniform when the Cartier indices are uniformly bounded.
\end{abstract}

\maketitle
\tableofcontents

\section{Introduction}
We work over the field of complex numbers $\mathbb C$.

We adopt the standard notation and definitions in \cite{KM98} and will freely use them. We say a pair $(X,B)$ is nef/anti-nef/pseudo-effective/anti-pseudo-effective if $K_X+B$ is nef/anti-nef/pseudo-effective/anti-pseudo-effective.

\emph{Rational polytopes} are natural objects to consider in birational geometry. Given a property $\mathcal P$ for pairs, it is natural to ask the following:

\begin{ques}\label{ques:open set of rational affine space in a family}
Let $\{(X_s,\sum_{i=1}^nB_{s,i})\}_{s\in S}$ be a set of couples of fixed dimension $d$, where $B_{s,i}$ are distinct prime divisors on $X_s$. Let $\bm{b}:=(b_1,\dots,b_n)\in\mathbb R_{>0}^n, B_s(\bm{v}):=\sum_{i=1}^nv_iB_{s,i}$ for any $\bm{v}=(v_1,\dots,v_n)\in\mathbb R^n$, and $V\subset\mathbb R^n$ the minimal rational affine subspace which contains $\bm{b}$. Assume that $(X_s,B_s(\bm{b}))$ is a pair satisfying $\mathcal P$ for any $s\in S$, does there exist an open set $U\ni\bm{b}$ of $V$, such that $(X_s,B_s(\bm{v}))$ satisfies $\mathcal P$ for any $\bm{v}\in U$ and $s\in S$?
\end{ques}

This question is important when we consider $\mathbb R$ pairs $(X,B)$ satisfying certain properties as we expect to write $K_X+B=\sum_{j=1}^{m}a_j(K_X+B_j)$, where $a_j>0$ and $(X,B_j)$ are $\mathbb Q$ pairs satisfying the same properties. Thus in some sense we need the pair $(X,B)$ to be contained in the rational polytope determined by $(X,B_j)$.

Question \ref{ques:open set of rational affine space in a family} is already very interesting when $S$ has only one element, and we can ask for a stronger version in this situation.
\begin{ques}\label{ques: rational polytope}
Let $(X,\sum_{i=1}^nb_iB_i)$ be a pair satisfying $\mathcal P$, where $B_i$ are prime divisors on $X$. Let $B(\bm{v}):=\sum_{i=1}^nv_iB_i$ for any $\bm{v}=(v_1,\dots,v_n)\in\mathbb R^n$. Is the set $\{\bm{v}\in\mathbb R^n~|~(X,B(\bm{v}))~\text{satisfies $\mathcal P$}\}$ a rational polytope?
\end{ques}

Question \ref{ques: rational polytope} has been extensively studied by many people and has a lot of important applications in birational geometry: 
\begin{enumerate}    
    \item When $\mathcal P$ is ``nef and klt/lc", a positive answer is given in \cite{BCHM10,Bir11,Sho92} and is essential to the proof of the existence of flips, finiteness of log terminal models and termination of flips. This result is also useful to show the finiteness of ample models \cite{Jia20}.
    \item When $\mathcal P$ is ``nef and lc" or ``anti-nef and lc", a positive answer for the generalized pair version of Question \ref{ques: rational polytope} is given in \cite{HL19,HL20} and is essential to prove the canonical bundle formula and the surface numerical non-vanishing conjecture for generalized pairs. This is also useful for results on the termination of flips \cite{HL18}.
\end{enumerate}

Question \ref{ques:open set of rational affine space in a family} is important when the pairs involved are not fixed and we want certain uniform results for those pairs. Recently, there are many positive results for Question \ref{ques:open set of rational affine space in a family} when
\begin{enumerate}
 \item $\mathcal P$ is ``lc/klt", ``$\delta$-lc for some $\delta>0$" ``$\delta$-klt for some $\delta>0$" (\cite[Thm 5.6]{HLS19}), 
 \item $\mathcal P$ is ``Fano type and $\Rr$-complementary" (\cite[Thm 5.16]{HLS19}).
 \end{enumerate}
 These results are essential to prove Shokurov's conjecture on the existence of $n$-complements. They have also been applied to the study on the ACC conjecture for minimal log discrepancies \cite{Kaw14,Nak16,Liu18,HLS19,HLQ20} and lc-trivial fibrations \cite{Li20}. Similar results can also be found in \cite{Chen20} and \cite{CX20} for generalized pairs. 
 
 The author was informed by Han and Liu that by using exactly the same method in their paper \cite[Chapter 5]{HLS19}, they could give a positive answer to Question \ref{ques:open set of rational affine space in a family} when 
 
 \begin{enumerate}
 \item[(3)] $\mathcal P$ is ``lc and pseudo-effective" (\cite{HL21}), and when
 \item[(4)] $\mathcal P$ is ``lc Fano type and anti-pseudo-effective" (\cite{HL21}).
\end{enumerate}

The key point of these result is that the existence of the open subset $U\ni\bm{b}$ does not depend on $X$. 

In this note, however, we give a surface counterexample to Question \ref{ques:open set of rational affine space in a family} when $\mathcal P$ is ``anti-nef and lc", which gives a negative answer to \cite[Problem 7.16]{CH20}. To be concise, the boundedness of Cartier index cannot be removed from the assumption in \cite[Thm 5.18]{CH20}. This example shows that the introduction of the condition ``$\mathbb R$-complementary" in \cite[Thm 1.10]{HLS19} is necessary. We also give another surface counterexample to Question \ref{ques:open set of rational affine space in a family} when $\mathcal P$ is ``nef and lc":

\begin{thm}\label{thm: rational nef polytope is not uniform}
Question \ref{ques:open set of rational affine space in a family} fails when $\mathcal P$ is ``nef and lc" (or ``anti-nef and lc") even when $\dim X_s=2$ for any $s\in S$. Furthermore, we can pick $(X_s,B_s(\bm{b}))$ such that $K_{X_s}+B_s(\bm{b})$ is ample (or anti-ample) and $(X_s,B_s(\bm{b}))$ is klt.
\end{thm}

Under certain boundedness conditions, we do have some positive results when $\mathcal P$ is``nef and lc":

\begin{thm}\label{thm: uniform polytope when Cartier Index is bdd}
Question \ref{ques:open set of rational affine space in a family} holds when $\mathcal P$ is ``nef and lc" if we further assume that the Cartier indices of $K_s,B_{s,i}$ are uniformly bounded.
\end{thm}

\medskip

\noindent\textbf{Acknowledgement}.  The author would like to thank Christopher D. Hacon for useful discussions and encouragements. The author would like to Jingjun Han, Jihao Liu, and Qingyuan Xue for useful discussions and suggestions. The author was partially supported by NSF research grants no: DMS-1801851, DMS-1952522 and by a grant from the Simons Foundation; Award Number: 256202.

\section{Example}
In this chapter we will use the notations and definitions of toric varieties from \cite{CLS10}.

\begin{ex}\label{Construction of toric surface}
Assume that $m>n\ge1$ are two integers. Let $e_1=(1,0),~e_2=(0,1)\in \mathbb R^2$, and $l=(-1,-m+1)=-e_1-(m-1)e_2,~u=(1,m)=e_1+me_2~,v=(m,1)=me_1+e_2,~w=(1,n)=e_1+ne_2$. Then $l,u,v,w$ are all primitive. 

Let $\Sigma_1$ be the complete fan in $N_{\mathbb R}=\mathbb R^2$ generated by rays $l,u,v$, and let $Y=X_{\Sigma_1}$ be the toric surface defined by $\Sigma_1$. Let $R_l,R_u,R_v$ be the torus invariant divisors on $Y$ corresponding to rays $l,u,v$ (cf. \cite[Thm 3.2.6]{CLS10}). 

Let $\Sigma_2$ be the complete fan in $N_{\mathbb R}=\mathbb R^2$ generated by rays $l,u,v,w$, and let $Z=X_{\Sigma_2}$ be the toric surface defined by $\Sigma_2$. Let $D_l,D_u,D_v,D_w$ be the torus invariant divisors on $Z$ corresponding to rays $l,u,v,w$. 

Let $f:Z\to Y$ be the induced toric morphism, then 
\begin{enumerate}
    \item $D_w^2<0$ and $-D_w$ is $f$-ample. 
    \item $R_l+R_u+R_v$ is ample.
    \item $F:=D_l+D_u+D_v$ is ample.
    \item $E:=D_l+D_u+D_v+\frac{n+1}{m+1}D_w=f^*(R_l+R_u+R_v)$ is semi-ample.
\end{enumerate}
\end{ex}
\begin{proof}
For (1), since $\Sigma_2$ is the star subdivision of $\Sigma_1$ at $w$, $D_w$ is the only exceptional divisor of $f$ (cf. \cite[Prop 11.1.6]{CLS10}). Therefore $D_w^2<0$ and $-D_w$ is $f$-ample. 

For (2) and (3), $D_l+D_u+D_w$ and $R_l+R_u+R_v$ are ample because their associated support functions are strictly convex (cf. \cite[Thm 6.1.14]{CLS10}).

For (4), notice that $w=\frac{mn-1}{m^2-1}u+\frac{m-n}{m^2-1}v$, so we have
$$f^*R_l=D_l,~f^*R_u=D_u+\frac{mn-1}{m^2-1}D_w~\text{and}~f^*R_v=D_v+\frac{m-n}{m^2-1}D_w$$
thus the equality holds (cf. \cite[Prop 6.2.7]{CLS10}). Since $R_l+R_u+R_v$ is ample, $E=f^*(R_l+R_u+R_v)$ is semi-ample.
\end{proof}

\begin{lem}\label{lem: preparation for anti-nef}
In Example \ref{Construction of toric surface}, let $a$ be a real number, then 
\begin{enumerate}
    \item $-(K_Z+aD_w)$ is ample if $\frac{m-n}{m+1}<a\le 1$.
    \item $-(K_Z+aD_w)$ is nef if $a=\frac{m-n}{m+1}$.
    \item $-(K_Z+aD_w)$ is not nef if $a<\frac{m-n}{m+1}$.
\end{enumerate}
\end{lem}
\begin{proof}
Notice that $K_Z\sim D_l+D_u+D_v+D_w$ by \cite[Thm 8.2.3]{CLS10}, thus we have $E=-(K_Z+\frac{m-n}{m+1}D_w)$ and $F=-(K_Z+D_w)$. Now (1) and (2) follow directly by (3)(4) of Example \ref{Construction of toric surface}. 

For (3), since $E\cdot D_w=0$ and , we have $-(K_Z+aD_w)\cdot D_w=(\frac{m-n}{m+1}-a)D_w^2$. Thus the statement follows by (1) of Example \ref{Construction of toric surface}.
\end{proof}

\begin{lem}\label{lem: preparation for nef}
In Example \ref{Construction of toric surface}, let $a$ be a real number, and $M$ be an integer bigger than $m+1$, then 
\begin{enumerate}
    \item $K_Z+aD_w+ME$ is ample if $0\le a<\frac{m-n}{m+1}$.
    \item $K_Z+aD_w+ME$ is nef if $0\le a=\frac{m-n}{m+1}$.
    \item $K_Z+aD_w+ME$ is not nef if $a>\frac{m-n}{m+1}$.
\end{enumerate}
\end{lem}

\begin{proof}
$K_Z+ME\sim (M-\frac{m+1}{n+1})E+\frac{m-n}{n+1}F$ is ample by (3)(4) of Example \ref{Construction of toric surface}, and  $K_Z+\frac{m-n}{m+1}D_w+ME\sim(M-1)E$ is nef. Thus (1) and (2) are clear.

Since $E\cdot D_w=0$, $(K_Z+aD_w+ME)\cdot D_w=(a-\frac{m-n}{m+1})D_w^2$. Therefore the statement follows by (1) of Example \ref{Construction of toric surface}.
\end{proof}

\begin{ex}\label{counterexample for "anti-nef" }
For any irrational number $r\in(0,1)$, we can select integers $m_i>n_i>1,~i\in\mathbb N^*$ such that $\frac{m_i-n_i}{m_i+1}<r$ and $\lim_{i\to\infty}\frac{m_i-n_i}{m_i+1}=r$. Then we can construct $l_i,u_i,v_i,w_i,Y_i,Z_i,F_i,E_i$ by using $(m_i,n_i)$ as in Example \ref{Construction of toric surface}. Now we have $-(K_{Z_i}+rD_{w_i})$ is ample by Lemma \ref{lem: preparation for anti-nef}(1), and $(Z_i,rD_{w_i})$ is klt since $r<1$. But by Lemma \ref{lem: preparation for anti-nef}(3), $\forall~ \epsilon\in(0,r)$, $(Z_i,(r-\epsilon)D_{w_i})$ is klt and $-K_{Z_i}-(r-\epsilon)D_{w_i}$ is not nef for $i\gg0$.
\end{ex}

\begin{ex}\label{counterexample for "nef"}
For any irrational number $r\in(0,1)$, we can select integers $m_i>n_i>1,~i\in\mathbb N^*$ such that $\frac{m_i-n_i}{m_i+1}>r$ and $\lim_{i\to\infty}\frac{m_i-n_i}{m_i+1}=r$. Then we can construct $l_i,u_i,v_i,w_i,Y_i,Z_i,F_i,E_i$ by using $(m_i,n_i)$ as in Example \ref{Construction of toric surface}. We can also choose $M_i$ sufficiently divisible such that $M_iE_i$ is Cartier and base point free. Let $B_i,C_i$ be two different general elements of $|M_iE_i|$ such that $(Z_i,D_w+B_i+C_i)$ is dlt, then $(Z_i,rD_w+\frac{1}{2}B_i+\frac{1}{2}C_i)$ is klt and $K_{Z_i}+rD_w+\frac{1}{2}B_i+\frac{1}{2}C_i\sim K_{Z_i}+rD_w+M_iE_i$ is ample by Lemma \ref{lem: preparation for nef}(1). But by Lemma \ref{lem: preparation for nef}(3), $\forall~\epsilon\in(0,1-r)$, $K_{Z_i}+(r+\epsilon)D_w+\frac{1}{2}B_i+\frac{1}{2}C_i$ is not nef for $i\gg 0$ and $(Z_i,(r+\epsilon)D_w+\frac{1}{2}B_i+\frac{1}{2}C_i)$ is klt.
\end{ex}

\begin{proof}[Proof of Theorem \ref{thm: rational nef polytope is not uniform}]
Example \ref{counterexample for "anti-nef" } gives a negative answer to Question \ref{ques:open set of rational affine space in a family} when $\mathcal P$ is ``anti-nef and lc", and Example \ref{counterexample for "nef"} gives a negative answer to Question \ref{ques:open set of rational affine space in a family} when $\mathcal P$ is ``nef and klt". In both cases the minimal rational affine subspace $V$ is of dimension $1$.

\end{proof}

\section{Positive results}

\begin{defn}
A curve $C$ on $X/Z$ is called extremal if it generates an extremal ray $R$ of $\overline{\mathrm{NE}}(X/Z)$ which defines a contraction $X \to Y /Z$ 
and if for some ample divisor $H$, we have $H\cdot C = \min\{H\cdot\Sigma~|~\Sigma\in[R]\}$,
where $\Sigma$ ranges over curves generating $R$.
\end{defn}


The following theorem is somewhat the ``nef" version of \cite[Lemma 5.18]{CH20}:
\begin{thm}[=Theorem \ref{thm: uniform polytope when Cartier Index is bdd}]
We follow the notations and assumptions in Question \ref{ques:open set of rational affine space in a family}. Further assume that for any $s\in S$ :
\begin{enumerate}
    \item $f_s: X_s\to Z_s$ are projective contractions.
    \item The Cartier indices of $K_{X_s},B_{i,s}$ are uniformly bounded for $s\in S$.
    \item $(X_s, B_s(\bm{b}))$ is lc and $K_{Z_s}+B_s(\bm{b})$ is nef.
\end{enumerate}  
Then there exists an open set $U\ni\bm{b}$ of $V$, such that for any $\bm{v}\in U$ and $s\in S$ we have $(X_s, B_s(\bm{v}))$ is lc and $K_{Z_s}+B_s(\bm{v})$ is nef. 
\end{thm}

\begin{proof}
We can assume that $\dim V=v>0$ otherwise we can choose $U=V=\{\bm{b}\}$. By \cite[Thm 5.6]{HLS19}, there exists $\bm{b}_1,\cdots,\bm{b}_{v+1}\in\mathbb Q^n\cap V$ and $a_1,\cdots,c_{v+1}>0$ such that,
\begin{enumerate}
    \item $\sum_{j=1}^{v+1}a_j=1$.
    \item the convex hull of $\{\bm{b}_1,\cdots,\bm{b}_{v+1}\}$ contains an open neighborhood of $\bm{b}$ in $V$.
    \item $\sum_{j=1}^{v+1}a_j(K_{X_s}+B_s(\bm{b}_j))=K_{X_s}+B_s(\bm{b})$ for any $s\in S$.
    \item $(X_s, B_s(\bm{b}_j))$ is lc for any $1\le j\le v+1$ and any $s\in S$.
    
\end{enumerate}

If $1>K_{X_s}+B_s(\bm{b})\cdot\Gamma_s>0$ for some extremal curve $\Gamma_s$ on $X_s/Z_s$, then we claim that we can find a number $1>\alpha>0$ independent of $s$ such that $(K_{X_s}+B_s(\bm{b}))\cdot\Gamma_s>\alpha$. Indeed, if 
$$
1\ge(K_{X_s}+B_s(\bm{b}))\cdot\Gamma_s=\sum_{j=1}^{v+1}a_j(K_{X_s}+B_s(\bm{b}_j))\cdot\Gamma_s,
$$
then each $a_j(K_{X_s}+B_s(\bm{b}_j))\cdot\Gamma_s$ is bounded from above and below since $(K_{X_s}+B_s(\bm{b}_j))\cdot\Gamma_s\ge-2d$ for any $1\le j\le v+1$ by the cone theorem. 

Since the Cartier indices of $K_{X_s},B_{s,i}$ are uniformly bounded, there are only finitely many possibilities of numbers $a_j(K_{X_s}+B_s(\bm{b}_j))\cdot\Gamma_s$ and the existence of $\alpha$ is clear. 

We claim that $K_{X_s}+B_s(\bm{u}_j)$ is nef for any $1\le j\le v+1$, where $\bm{u}_j=(1-\frac{\alpha}{4d})\bm{b}+\frac{\alpha}{4d}\bm{b}_j$. Otherwise there exists an extremal curve $\Gamma_s$ on $X_s/Z_s$ such that $K_{X_s}+B_s(\bm{u}_j)\cdot\Gamma_s<0$, then $(K_{X_s}+B_s(\bm{b}))\cdot\Gamma_s<\alpha$ since $(K_{X_s}+B_s(\bm{b}_j))\cdot\Gamma_s\ge-2d$. Thus $(K_{X_s}+B_s(\bm{b}))\cdot\Gamma_s=0$ because $K_{X_s}+B_s(\bm{b})$ is nef. Since $V$ is the minimal rational affine subspace containing $\bm{b}\in\mathbb R^n$, $(K_{X_s}+B_s(\bm{u}))\cdot\Gamma_s=0$ for any $\bm{u}\in V$, which contradicts $\bm{u}_j\in V$.

Now the statement follows since the convex hull of $\{\bm{u}_1,\cdots,\bm{u}_{v+1}\}$ contains an open neighborhood of $\bm{b}$ in $V$.

\end{proof}

\begin{rem}
It's also worth mentioning that positive answers can be given to Question \ref{ques:open set of rational affine space in a family} when $\mathcal P$ is ``nef and lc" or ``ant-nef and lc" if  $\rho(X_s)=1 $ for any $s\in S$ in \cite[Thm 3.8]{Nak16} and \cite[7.16]{CH20}.
\end{rem}

\end{document}